\documentclass[12pt]{amsart}

\normalfont
\usepackage[T1]{fontenc}

\usepackage{amssymb}
\usepackage{graphicx}
\usepackage{lscape}
\usepackage{mathrsfs}

\usepackage[curve,tips]{xypic}
\SelectTips{eu}{11} \UseTips

\renewcommand{\epsilon}{\varepsilon}

\newtheorem{theorem}{Theorem}[section]
\newtheorem{prop}[theorem]{Proposition}
\newtheorem{corollary}[theorem]{Corollary}
\newtheorem{lemma}[theorem]{Lemma}

\newtheorem*{theoremA}{Theorem A}

\theoremstyle{definition}

\theoremstyle{remark}

\newcommand{\Q}{\mathbb Q}

\newcommand{\Z}{\mathbb Z}
\newcommand{\N}{\mathbb N}

\newcommand{\FP}{\operatorname{FP}}

\newcommand{\fpinfty}{{\FP}_{\infty}}

\newcommand{\cohom}[3]{H^{{\raise1pt\hbox{$\scriptstyle#1$}}}(#2\>\!,#3)}
\newcommand{\tatecohom}[3]{\widehat H^{{\raise1pt\hbox{$\scriptstyle#1$}}}(#2\>\!,#3)}

\newcommand{\Cohom}[3]%
  {H^{{\raise1pt\hbox{$\scriptstyle#1$}}}\big(#2\>\!,#3\big)}
\newcommand{\Tatecohom}[3]%
  {\widehat H^{{\raise1pt\hbox{$\scriptstyle#1$}}}\big(#2\>\!,#3\big)}

\newcommand{\homol}[3]{H_{{\lower1pt\hbox{$\scriptstyle#1$}}}(#2\>\!,#3)}
\newcommand{\homolog}[2]{H_{{\lower1pt\hbox{$\scriptstyle#1$}}}(#2)}

\newcommand{\epi}{\twoheadrightarrow}

\newcommand{\Fit}{\operatorname{Fitt}}

\title{Finitary Group Cohomology and Group Actions on Spheres}
\author{Martin Hamilton}
\address{Department of Mathematics, University of Glasgow, University Gardens, Glasgow G12 8QW, United Kingdom}
\email{m.hamilton@maths.gla.ac.uk}

\subjclass[2000]{20J06 18A22 20E34}

\keywords{cohomology of groups, finitary functors, group actions
on spheres}

\begin{document}

\begin{abstract}

We show that if $G$ is an infinitely generated locally
(polycyclic-by-finite) group with cohomology almost everywhere
finitary, then every finite subgroup of $G$ acts freely and
orthogonally on some sphere.

\end{abstract}

\maketitle

\section{Introduction}

In \cite{Finitarycohom} the question of which locally
(polycyclic-by-finite) groups have cohomology almost everywhere
finitary was considered. Recall that a functor is \emph{finitary}
if it preserves filtered colimits (see \S 6.5 in \cite{Leinster};
also \S 3.18 in \cite{AccessibleCategories}). The $n$th cohomology
of a group $G$ is a functor $H^n(G,-)$ from the category of $\Z
G$-modules to the category of abelian groups. If $G$ is a locally
(polycyclic-by-finite) group, then Theorem 2.1 in
\cite{Continuity:cohomologyfunctors} shows that the \emph{finitary
set}
$$\mathscr{F}(G):=\{n\in\N:H^n(G,-)\mbox{ is finitary}\}$$ is
either cofinite or finite. If
  $\mathscr{F}(G)$ is cofinite, we say that $G$ has
  \emph{cohomology almost everywhere finitary}, and if
  $\mathscr{F}(G)$ is finite, we say that $G$ has \emph{cohomology
  almost everywhere infinitary}.

We proved the following results about locally
(polycyclic-by-finite) groups with cohomology almost everywhere
finitary in \cite{Finitarycohom}:

\begin{theorem}\label{theorem A from previous}
  Let $G$ be a locally (polycyclic-by-finite) group. Then $G$ has
  cohomology almost everywhere finitary if and only if $G$ has
  finite virtual cohomological dimension and the normalizer of
  every non-trivial finite subgroup is finitely generated.
\end{theorem}

\begin{corollary}\label{cor B from previous}
  Let $G$ be a locally (polycyclic-by-finite) group with
  cohomology almost everywhere finitary. Then every subgroup of
  $G$ also has cohomology almost everywhere finitary.
\end{corollary}

Recall (see, for example, \cite{ThomasWall1}) that a finite group
acts freely and orthogonally on some sphere if and only if every
subgroup of order $pq$, where $p$ and $q$ are prime, is cyclic. In
this paper, we prove the following result:

\begin{theoremA}
  Let $G$ be an infinitely generated locally
  (polycyclic-by-finite) group with cohomology almost everywhere
  finitary. Then every finite subgroup of $G$ acts freely and
  orthogonally on some sphere.
\end{theoremA}

Note that we cannot remove the "infinitely generated" restriction;
as, for example, every finite group is of type $\fpinfty$ and so
has $n$th cohomology functors finitary for all $n$, by a result of
Brown (Corollary to Theorem $1$ in \cite{BrownPaper}).

\subsection{Acknowledgements}

I would like to thank my research supervisor Peter Kropholler for
all the help he has given me with this paper. I would also like to
thank Frank Quinn for suggesting that a result like Theorem A
should be true.

\section{Proof}

The following Proposition sets the scene for proving Theorem A:

\begin{prop}\label{soluble-by-finite result}
  Let $G$ be a locally (polycyclic-by-finite) group with
  cohomology almost everywhere finitary. Then $G$ has a
  characteristic subgroup $S$ of finite index, such that $S$ is
  torsion-free soluble of finite Hirsch length.
\end{prop}

\begin{proof}
  By Theorem 2.1 in
\cite{Continuity:cohomologyfunctors} we know that there is a
finite-dimensional contractible $G$-CW-complex $X$ on which $G$
acts with finite isotropy groups, and that there is a bound on the
orders of the finite subgroups of $G$.

We know that the rational cohomological dimension of $G$ is
bounded above by the dimension of $X$ (see, for example,
\cite{Deltapaper}), so $G$ has finite rational cohomological
dimension. Recall that the class of elementary amenable groups is
the class generated from the finite groups and $\Z$ by the
operations of extension and increasing union (see, for example,
\cite{HillmanLinnell}), so $G$ is elementary amenable. According
to \cite{HillmanLinnell}, the Hirsch length of an elementary
amenable group is bounded above by its rational cohomological
dimension, so $G$ has finite Hirsch length.

Let $\tau(G)$ denote the join of the locally finite normal
subgroups of $G$. As there is a bound on the orders of the finite
subgroups of $G$, this implies that $\tau(G)$ is finite. Replacing
$G$ with $G/\tau(G)$, we may assume that $\tau(G)=1$.

Now $G$ is an elementary amenable group of finite Hirsch length,
so it follows from a minor extension of a theorem by Mal'cev (see
Wehrfritz's paper \cite{WehrfritzElAm}) that $G/\tau(G)=G$ has a
poly (torsion-free abelian) characteristic subgroup of finite
index.

\end{proof}

Before proving Theorem A, we need four Lemmas:

\begin{lemma}\label{abelian-by-finite case}
  Let $Q$ be a non-cyclic group of order $pq$, where $p$ and $q$
  are prime, and let $A$ be a $\Z$-torsion-free $\Z Q$-module such
  that the group $A\rtimes Q$ has cohomology almost everywhere
  finitary. Then $A$ is finitely generated.
\end{lemma}

\begin{proof}
  We write $G:=A\rtimes Q$.

  For any $K\leq Q$ we write $\widehat{K}$ for the element of $\Z
  Q$ given by $$\widehat{K}:=\sum_{k\in K} k.$$ Notice that
  $\widehat{K}.A$ is contained in the set of $K$-invariant
  elements $A^K$ of $A$.

  There are two cases to consider:

  If $Q$ is abelian, then $p=q$ and $Q$ has $p+1$ subgroups
  $E_0,\ldots,E_p$ of order $p$. We have the following equation in
  $\Z Q$: $$\sum_{i=0}^p \widehat{E_i}=\widehat{Q}+p.1,$$ so it
  follows that for any $a\in A$ $$p.a=\sum_{i=0}^p
  \widehat{E_i}.a-\widehat{Q}.a\in\sum_{i=0}^p A^{E_i}+A^Q$$ and
  hence $$p.A\subseteq \sum_{i=0}^p A^{E_i}+A^Q.$$ If $K$ is
  non-trivial, then it follows from Theorem \ref{theorem A from
  previous} that $N_G(K)$ is finitely generated. Then, as $A^K\leq
  N_G(K)$, it follows that $A^K$ is also finitely generated. Hence
  we see that $p.A$ is finitely generated, and as $A$ is
  torsion-free, we conclude that $A$ is finitely generated.

  If $Q$ is non-abelian, then $p\neq q$, and without loss of generality we may assume that $p<q$. Then
  $Q$ has one subgroup $F$ of order $q$ and $q$ subgroups
  $H_0,\ldots,H_{q-1}$ of order $p$. We have the following
  equation in $\Z Q$:
  $$\sum_{i=0}^{q-1}\widehat{H_i}+\widehat{F}=\widehat{Q}+q.1$$
  and the proof continues as above.
\end{proof}

Recall (see, for example, \S 10.4 in \cite{Rob2}) that a group $G$
is \emph{upper-finite} if and only if every finitely generated
homomorphic image of $G$ is finite. The class of upper-finite
groups is closed under extensions and homomorphic images. Also
recall (see \S 10.4 in \cite{Rob2}) that the \emph{upper-finite
radical} of a group $G$ is the subgroup generated by all of its
upper-finite normal subgroups, and is itself upper-finite.

\begin{lemma}\label{upp-fin tensor product lemma}
  Let $A$ and $B$ be abelian groups. If $A$ is upper-finite, then
  $A\otimes B$ is upper-finite.
\end{lemma}

\begin{proof}
  If $b\in B$, then $A\otimes b$ is a homomorphic image of $A$ and
  hence is upper-finite. Then as $A\otimes B$ is generated by all
  the $A\otimes b$ it is also upper-finite.
\end{proof}

\begin{lemma}\label{derived subgp of upp-fin nilp gp}
  Let $G$ be an upper-finite nilpotent group. Then its derived
  subgroup $G'$ is also upper-finite.
\end{lemma}

\begin{proof}
  As $G$ is upper-finite, it follows that $G/G'$ is also
  upper-finite.

  As $G$ is a nilpotent group, it has a finite lower central
  series
  $$G=\gamma_1(G)\geq\gamma_2(G)\geq\cdots\geq\gamma_k(G)=1,$$
  where $\gamma_2(G)=G'$.

  For each $i$ there is an epimorphism $$\underbrace{G/G'\otimes\cdots\otimes
  G/G'}_i\epi\gamma_i(G)/\gamma_{i+1}(G),$$ and as $\underbrace{G/G'\otimes\cdots\otimes
  G/G'}_i$ is upper-finite, from Lemma \ref{upp-fin tensor product
  lemma}, we see that each $\gamma_i(G)/\gamma_{i+1}(G)$ is
  upper-finite. Then, as the class of upper-finite groups is
  closed under extensions, we conclude that $G'$ is also
  upper-finite.

\end{proof}

\begin{lemma}\label{centre of nilp fg implies gp fg}
  Let $G$ be a torsion-free nilpotent group of finite Hirsch length. If the
  centre $\zeta(G)$ of $G$ is finitely generated, then $G$ is finitely
  generated.
\end{lemma}

\begin{proof}
  Let $K$ be the upper-finite radical of $G$. As $G$ is torsion-free nilpotent of finite Hirsch length, it is a
  special case of Lemma 10.45 in \cite{Rob2} that $G/K$ is
  finitely generated. Suppose that $K\neq 1$.

  Following an argument of Robinson (Lemma 10.44 in \cite{Rob2}) we see that for each $g\in G$, $[K,g]K'/K'$ is a homomorphic image of $K$, and so is
  upper-finite, so therefore $[K,G]/K'$ is upper-finite. Then, as $K'$ is
  upper-finite, from Lemma \ref{derived subgp of upp-fin nilp gp},
  we see that $[K,G]$ is also upper-finite. Similarly, we see by induction
  that $[K,{}^mG]=[K,\underbrace{G,\ldots,G}_m]$ is upper-finite.

  Choose the largest $m$ such that $[K,{}^mG]\neq 1$. Then
  $[K,{}^mG]\subseteq\zeta(G)$, so $[K,{}^mG]$ is finitely
  generated, and hence finite. Then, as $G$ is torsion-free, we
  see that $[K,{}^mG]=1$, which is a contradiction. Therefore,
  $K=1$, and so $G$ is finitely generated.

\end{proof}

We can now prove Theorem A.

\begin{proof}[Proof of Theorem A]

Let $G$ be an infinitely generated locally (polycyclic-by-finite)
group with cohomology almost everywhere finitary. It follows from
Proposition \ref{soluble-by-finite result} that $G$ has a
characteristic subgroup $S$ of finite index such that $S$ is
torsion-free soluble of finite Hirsch length.

Suppose that not every subgroup of $G$ acts freely and
orthogonally on some sphere, so there is a non-cyclic subgroup $Q$
of order $pq$, where $p$ and $q$ are prime.

As $S$ is a torsion-free soluble group of finite Hirsch length, it
is linear over the rationals (see, for example, \cite{Wehrfritz}),
so by a result of Gruenberg (see Theorem 8.2 of \cite{Wehrfritz})
the Fitting subgroup $F:=\Fit(S)$ of $S$ is nilpotent. Now the
centre $\zeta(F)$ of $F$ is a characteristic subgroup of $G$, so
we can consider the group $\zeta(F)Q$. It then follows from
Corollary \ref{cor B from previous} that $\zeta(F)Q$ has
cohomology almost everywhere finitary. Then, by Lemma
\ref{abelian-by-finite case}, we see that $\zeta(F)$ is finitely
generated. It then follows from Lemma \ref{centre of nilp fg
implies gp fg} that $F$ is finitely generated.

Now, let $K$ be the subgroup of $S$ containing $F$ such that
$K/F=\tau(S/F)$. As $S$ is linear over $\Q$, we see that $S/F$ is
also linear over $\Q$, and as locally finite $\Q$-linear groups
are finite (see, for example, Theorem 9.33 in \cite{Wehrfritz}),
we conclude that $K/F$ is finite. An argument of Zassenhaus in
15.1.2 of \cite{RobGroup} shows that $S/K$ is maximal
abelian-by-finite; that is, crystallographic. Hence $S/F$ is
finitely generated, so we conclude that $S$ is finitely generated,
a contradiction.

\end{proof}

\bibliographystyle{amsplain}

\bibliography{spheres}

\end{document}